\documentclass[11pt]{article}
\usepackage{amsfonts}
\usepackage[mathscr]{eucal}
\usepackage{mathrsfs}
\usepackage{amsmath}
\usepackage{amssymb}

\newtheorem{theorem}{Theorem}[section]
\newtheorem{lemma}[theorem]{Lemma}
\newtheorem{proposition}[theorem]{Proposition}
\newtheorem{definition}[theorem]{Definition}
\newtheorem{remark}[theorem]{Remark}
\newtheorem{corollary}[theorem]{Corollary}
\newtheorem{example}[theorem]{Example}

\newcommand{\Om}{{\Omega}}

\newcommand{\al}{{\alpha}}

\newcommand{\la}{\lambda}

\newcommand{\vr}{{\varrho}}

\newcommand{\sB}{\mathscr{B}}

\newcommand{\sD}{\mathscr{D}}

\newcommand{\R}{{\mathbb R}}

\newcommand{\cO}{{\cal O}}

\newcommand{\cM}{{\cal M}}

\newcommand{\cR}{{\cal R}}
\newcommand{\cL}{{\cal L}}

\newcommand{\hf}{{\frac12}}

\newcommand{\g}{{\nabla}}

\newcommand{\pd}{\partial}

\newcommand{\di}{{\rm div\, }}
\newenvironment{declaration}[1]{\trivlist
\item[\hskip \labelsep{\bf #1 }]\ignorespaces}{\endtrivlist}
\newenvironment{proofof}[1]{\begin{declaration}{#1}}{\hfill
$\square$ \end{declaration}}
\newenvironment{proof}{\begin{proofof}{Proof.}}{\end{proofof}}
\begin{document}
\title{Discrete data assimilation  via \\ Ladyzhenskaya  squeezing property  in \\  the 3D viscous primitive equations
}
\author{Igor Chueshov\footnote{ Department of Mechanics and Mathematics,
 Karazin Kharkov National University,  Kharkov, 61022,  Ukraine, e-mail:
 chueshov@karazin.ua},
 }
\maketitle

\begin{abstract}
We discuss the discrete data assimilation problem for
 the  3D viscous primitive equations
 arising in the  modeling of  large scale phenomena
in oceanic dynamics. Our main result states possibility of asymptotically reliable
prognosis based on a discrete sequence of finite number of scalar observations.
Our method is quite general and can be applied to a wide class of dissipative systems. 
\end{abstract}

\section*{Introduction}
Data assimilation problem is a question how to incorporate available observation data  in
computational schemes to improve quality of predicting  of the future evolution of the corresponding
dynamical system.
This problem has a long history and was studied by many authors at different levels
(see, e.g., the monographs \cite{data-book,kal2003-data} and the references therein).
\par
In this paper we consider the case when observations of the system are making in  some sequence
$\{t_n\}$ of moments of time and use the same formulation of the data assimilation problem
as   in \cite{HO-Titi-data-as2011}.
\par
Our main goal is to demonstrate the role of the so-called  Ladyzhenskaya squeezing
property (which is valid for a wide class parabolic type PDEs,   see
\cite{lad1,lad2}) in the  solving of data assimilation problems.
As in \cite{HO-Titi-data-as2011} we
also involve the notion of determining modes or, more generally,
determining functionals.
However our method is different from the approach developed in \cite{HO-Titi-data-as2011}
for the 2D Navier--Stokes equations. Moreover, the method based on the squeezing property
looks more  general and can be applied to a wide class of dissipative systems
admitting the Ladyzhenskaya property.
\par
In this paper  as a model we choose the system
 of the
 3D viscous primitive equations
which arise in geophysical fluid dynamics  for modeling  large scale phenomena
in oceanic motions.
In this case  data assimilation problem is related with reliability of weather predictions.
 \par
 The paper is organized as follows. In Section~\ref{sec:3PE} we describe the model and quote its
 properties which need for data assimilation.
 The main technical tools are the Ladyzhenskaya squeezing
property and the theory of determining functionals which we discuss shortly in Section~\ref{sec:def-f}.
In Section~\ref{sec:data} formulate  the data assimilation problem,
introduce the notion of asymptotically reliable prognosis and  prove our main result
concerning a finite number of scalar observations in a discrete sequence of times.

\section{3D Primitive equations}\label{sec:3PE}

The primitive equations are based on the so-called hydrostatic approximation of the 3D
 Navier-Stokes equations for velocity field $u$ and
coupled to thermo- and salinity  equations which are taken into account
via  small  variation of density (or equivalently via buoyancy $b$),
see, e.g., the survey~\cite{temam2008-srw} and the literature cited there.
Below to simplify the presentation we consider
periodic boundary conditions of the same type as in \cite{petcu} and~\cite{temam2008-srw}
(see also \cite{Chueshov-squeszing2012}).
 However, it should be noted that
our results remains  valid in the case of  free type boundary conditions
 like in \cite{CaoTiti}.
The case of
 mixed free-Dirichlet boundary conditions (see \cite{kuka}) is more complicated
and requires a separate consideration.
\par
Let
 \[
 \cO= (0,L_1)\times (0,L_2)\times (-L_3/2,L_3/2) \subset \R^3.
 \]
We denote by $\bar{x}=(x,z)=(x_1,x_2,z)$ the spatial variable in $\cO$ and
suppose that $\g$, $\di$ and $\Delta$ are the gradient, divergence and  Laplace operators
in the (horizontal)  variable $x=(x_1,x_2)$.
\par
After the reduction based on  the  hydrostatic relation we
arrive (see, e.g., \cite{temam2008-srw}) at the  following  equations
for the horizontal  fluid velocity
\[
v(\bar{x},t)=(v^1(\bar{x},t);v^2(\bar{x},t)),~~ \bar{x}=(x,z),
\]
 for the (surface) pressure $p=p(x,t)$ and for the buoyancy $b= b(\bar{x},t)$:
 \begin{align}\label{fl.1a}
   v_t +(v,\g)v & -\left[\int_0^{z} \di vd\xi \right]\pd_z v -\nu [\Delta v +\pd_{zz} v]
 +f v^{\perp}  \notag
  \\ & =-\nabla \left[ p(x,t)+\int_0^{z} b d\xi\right]
 +G_f\quad {\rm in\quad} \cO
   \times(0,+\infty),
\end{align}
    \begin{equation}\label{fl.4a}
   b_t +(v,\g) b - \left[\int_0^{z} \di vd\xi \right] \pd_zb-\nu [\Delta b +\pd_{zz} b] =G_b\quad {\rm in\quad} \cO
   \times(0,+\infty),
  \end{equation}
 where $\nu>0$ is the dynamical viscosity, $f$ is the Coriolis parameter and
 $v^{\perp}=(-v^2;v^1)$. The functions represents  $G_f$ and $G_b$ are volume sources
 related to the fluid field and the buoyancy.
 As in   \cite{petcu,temam2008-srw}
the equations in (\ref{fl.1a}) and (\ref{fl.4a})
  supplied with the conditions:
   \begin{equation}\label{fl.4a-bc1}
 {\rm div}\,\int_{-L_3/2}^{L_3/2} v dz=0;~~
v ~\mbox{is periodic in $\bar{x}$ and even in $z$},~~\int_{\cO} v d\bar{x}=0;
  \end{equation}
   \begin{equation}\label{fl.4a-bc2}
  b ~\mbox{is periodic in $\bar{x}$ and odd in $z$}.
   \end{equation}
 We also impose  initial data:
   \begin{equation}\label{fl.4a-id}
 v(0)=v_0,~~~b(0)=b_0.
  \end{equation}
 We note that the vertical component of the velocity field has the form
 \begin{equation*}%\label{rep-w}
w(\bar x,t)=-\int_0^{z} \di v(x,\xi,t) d\xi~~\mbox{for every $\bar{x}=(x,z)\in \cO$}
\end{equation*}
and thus the full velocity field $(v^1,v^2,w)$ satisfies incompressibility condition.
We  emphasize  that the (surface) pressure $p$  in  (\ref{fl.1a}) depends
on 2D (horizontal) variable $x$ only. Basing on  this   observation
a new effective approach \cite{CaoTiti}  for proving of the global well-posedness for  problems like
(\ref{fl.1a}) and (\ref{fl.4a}) have been implemented, see \cite{CaoTiti} and the discussion therein.
\par
The  system of the viscous 3D primitive equations  was intensively
 studied for the different types of boundary conditions
 (see  the literature cited in  the survey \cite{temam2008-srw}). The existence  of week solutions
 was established in \cite{LTW-ocean}; for
  global well-posedness of strong solutions we refer to
 \cite{CaoTiti} and also to \cite{kobelkov,kuka,petcu}.
The uniqueness of weak solutions is still unknown.
The question on a global attractor for the viscous 3D primitive equations
was considered  in \cite{ju} (see also the papers \cite{petcu} and \cite{Chueshov-squeszing2012} devoted
to the periodic case).

\par
\medskip\par

We denote  by $\dot{H}_{per}^s(\cO)$ the Sobolev space of order $s$
consisting of periodic functions  such that $\int_{\cO} f d\bar{x}=0$ and  introduce the following spaces:
\par
\begin{equation*}%\label{V-space}
V_s=\left\{
v=(v^1;v^2)\in [\dot{H}_{per}^s(\cO)]^2 :\; v^i~\mbox{is even in}~z,\;
 {\rm div}\,\int_{-L_3/2}^{L_3/2} v dz=0
  \right\}
\end{equation*}
for $s\ge 0$.
We equip $H\equiv V_0$ with $L_2$-norm $\|\cdot\|$
and denote by $(\cdot,\cdot)$ the corresponding inner product.
It is convenient to endow the spaces $V_1$ and $V_2$  with the norms  $\|\cdot\|_{V_1}= \|\nabla_{x,z}\cdot\|$
and $\|\cdot\|_{V_2}= \|\Delta_{x,z}\cdot\|$.
Here and below we  use the notations $\g_{x,z}$  and $\Delta_{x,z}$
for gradient and Laplace operation in the 3D variable $(x,z)$.

\par
We also introduce  state spaces for  the buoyancy variable by the formulas
\begin{equation*}%\label{E-space}
E_s=\left\{
b\in \dot{H}_{per}^s(\cO) :\; b~\mbox{is odd in}~z
  \right\},~~~s\ge0.
\end{equation*}
We
equip them with the standard Sobolev norms. We  suppose $W_s=V_s\times E_s$
with the corresponding  (Hilbert) product norms.
\par
As it was already mentioned, starting with \cite{CaoTiti} the global well-posedness of
the equations  in  \eqref{fl.1a} and \eqref{fl.4a} was studied by many authors~\cite{kobelkov,kuka,petcu,temam2008-srw}.
The following
 result on well-posedness
 of strong solutions in  the case of periodic boundary conditions  was basically proved in
 \cite{petcu} (see also \cite{temam2008-srw} and Remark~2.2 in \cite{Chueshov-squeszing2012}).

\begin{proposition}[\cite{petcu}]\label{pr:titi}
Let $G_f\in V_0$, $G_b\in E_0$, and $U_0=(v_0;b_0)\in W_1$. Then  problem (\ref{fl.1a})--(\ref{fl.4a-id})
 has a unique strong solution $(v(t);b(t))$:
\begin{equation*}%\label{str-sol}
U(t;U_0)\equiv (v(t);b(t))\in C(\R_+; W_1)\cap L_2(0,T; W_2),~~~\forall T>0.
\end{equation*}
This solution generates a dynamical system $(S_t, W_1)$ with
the evolution operator $S_t$
defined by the relation $S_tU_0=U(t;U_0)$. The operator $S_t$
satisfies  the  Lipschitz property:
\[
\|S_tU-S_tU_*\|_{W_1}\le C_{T,\vr}\|U-U_*\|_{W_1},~~ t\in [0,T],
\]
for every $T>0$ and $U,U_*\in \sB_1(\vr)\equiv\{U : \|U\|_{W_1}\le \vr\}$.
\end{proposition}
 If
 $G_f\in V_{m-1}$, $G_b\in E_{m-1}$ and $U_0=(v_0;b_0)\in W_m$
for some $m\ge 2$ then (see \cite{petcu,temam2008-srw}) the solution $U$
lies in the class
$C(\R_+; W_m)\cap L^{loc}_2(\R_+; W_{m+1})$.
This observation makes it possible to
use smooth approximations of solutions in the  calculations
with multipliers (see, e.g., \cite{Chueshov-squeszing2012}).

For our goal the following assertion is important.

\begin{proposition}\label{pr:frame-for-kicks}
 Let the hypotheses of Proposition~\ref{pr:titi} be in force.
 Then
 \begin{itemize}
   \item There exist positive constants $a_0$ and $a_1$ such that
   \begin{equation}\label{S-sq1}
    \|S_tU\|\le e^{-a_0t}\|U\| +a_1 K_G,~~~ t\ge 0, ~~ U\in W_1,
   \end{equation}
   where $K_G^2=\|G_f\|^2+\|G_b\|^2$.
   \item If we assume in addition that
      \begin{equation}\label{GG-cond}
G_f\in V_1, ~G_b\in E_1~~ and~ also~~
   (\pd_zG_f;\pd_zG_b)\in \big[L_6(\cO)\big]^3,
\end{equation}
then for every $\vr>0$ and $0<\al\le\beta<+\infty$
   there exists the constant $C(\al,\beta, \vr)>0$ such that
\begin{equation}\label{S-sq2}
    \|S_tU\|_{W_2}\le C(\al,\beta, \vr)~~\mbox{for every}~ t\in [\al,\beta], ~~\|U\|\le \vr.
   \end{equation}

 \end{itemize}
\end{proposition}
\begin{proof}
  The first statement is achieved by the standard multipliers $v$ and $b$ applied to
  \eqref{fl.1a} and \eqref{fl.4a}, see \cite{petcu,temam2008-srw}, for instance.
  \par
  The second statement  is a  more complicated and based mainly on the calculations
given in \cite{CaoTiti} and \cite{petcu}. The corresponding argument involves the splitting
of the system into 2D Navier-Stokes type equations coupled  with 3D Burgers type model
(see \cite{CaoTiti} and also \cite{petcu,temam2008-srw})
and consists of several steps based on the application of the same multipliers as in
\cite{CaoTiti,petcu,temam2008-srw}. The spatial periodicity of the system allows us to use freely
higher order multipliers like $\Delta_{x,z}^2 v$ and $\Delta_{x,z}^2 b$.
For some related details we refer to the paper \cite{Chueshov-squeszing2012} which contains a very similar
argument in the proof of Theorem 3.1 on the existence of a smooth absorbing set.
\end{proof}
Proposition~\ref{pr:frame-for-kicks} implies that the system $(S_t,W_1)$ possesses
an absorbing set which is bounded in $W_2$. More precisely we have the following assertion.
\begin{corollary}[Smooth Absorbing Ball]\label{co:absorbing}
Let
\eqref{GG-cond}   be in force.
Then
the\-re exists $K>0$ such that the ball
$$
\sB\equiv \sB_{2}(K)=\{ U\in W_2: \|U\|_{W_2}\le K\}
$$
is absorbing for the dynamical
system $(S_t, W_1)$ generated by problem (\ref{fl.1a})--(\ref{fl.4a-id}), i.e.,
for any bounded set $B$ in $W_1$ there is $t_B$ such that
\[
S_tB\subset \sB ~~~\mbox{for all}~ t\ge t_B.
\]%
%Moreover for every $\rho\ge K$ the set
%$$
%\sD_\rho= \cup_{t\ge 0}S_t\{ U\in W_2: \|U\|_{W_2}\le \rho\}
%$$
% is a forward invariant  absorbing set which is bounded in $W_2$.
\end{corollary}
\begin{proof}
It follows from \eqref{S-sq1} that
\[
\|S_tU\|\le 1+ a_1 K_G~~\mbox{for all}~~ t\ge t_B,
\]
and thus by \eqref{S-sq2} we have that
\[
\|S_{t+1}U\|_{W_2} \le K\equiv C(1,1, 1 +a_1 K_G)~~\mbox{for all}~~ t\ge t_B,
\]
i.e. the ball $\sB$ possesses the desired property.
%\par
%To prove the second part of the statement we need only show that the set $\sD_\rho$
%is bounded in $W_2$. For this it is sufficient to show that the set
%$$
%\sD^T_\rho= \cup_{t\in [0,T]}S_t\{ U\in W_2: \|U\|_{W_2}\le \rho\}~~\mbox{is bounded in}~~ W_2
%$$
%for every fixed $T>0$. This property can be achieved by the same argument as in \cite{Chueshov-squeszing2012,petcu}
%with the help of the same multipliers.
\end{proof}
We can also prove the Lipschitz property in $H$ provided one of   two solutions belongs to $W_2$.
\begin{proposition}\label{pr:Lip-H}
Let $U_1(t)$ and  $U_2(t)$ be two strong  solutions
to  (\ref{fl.1a})--(\ref{fl.4a-id}). Assume that $\|U_1(t)\|_{W_2}\le R$ for all $t\in [0,T]$
for some $T>0$. Then
\begin{equation}\label{LipH-1}
\|  U_1(t)-U_1(t)\|\le C_T(R) \|  U_1(0)-U_1(0)\|,~~~\forall\, t\in [0,T].
\end{equation}
\end{proposition}
\begin{proof}
We note (see, e.g., \cite{temam2008-srw}) that problem  (\ref{fl.1a})--(\ref{fl.4a-id}) can be written in the form
\[
\pd_t U+\nu AU+B(U,U)+CU=G,~~~ U(0)=U_0,
\]
where
$A$ is a positive self-adjoint operator in $H$
 generated by  the bilinear form
\begin{equation}\label{a-form}
a(U,U_*)=\int_\cO \left[ \g_{x,z} v\cdot\g_{x,z} v_* + \g_{x,z} b\cdot\g_{x,z} b_*\right] dxdz,
\end{equation}
where $U=(v;b)$ and $U_*=(v_*;b_*)$ are  from $W_1$,  $C$ is a bounded skew-symmetric operator, and
$B(U,U)$ is a quadratic operator possessing the properties
\begin{equation}
 B(U^*,U),U)=  0, ~~~
  B(U,U^*),U)\le  C\|U\|^{3/2}_{W_1} \|U\|^{1/2} \|U^*\|_{W_2}
\label{LipH-2}
\end{equation}
for every $U^*\in W_2$ and $U\in W_1$. Thus for the difference $V=U_1(t)-U_2(t)$
we have that
\[
\hf\pd_t\|V(t)\|^2 +\nu a(V(t),V(t))+  B(V,U_1),V)=0
\]
which, via \eqref{LipH-2} and Gronwall's lemma, implies the relation in \eqref{LipH-1}.
\end{proof}

We note that the operator $A$ generated by   form  \eqref{a-form} has a discrete spectrum. This means that
there exists an orthonormal basis $\{ e_k\}$ in $H$ such  that
\begin{equation}\label{a-form2}
A e_k=\la_k e_k,~~~0<\la_1\le \la_2\le\ldots,~~~\lim_{k\to+\infty}\la_k=+\infty.
\end{equation}
We denote by $P_N$ the orthoprojector onto  Span$\{e_1,\ldots,e_N\}$
and $Q_N=I-P_N$.
\par
The following Ladyzhenskaya squeezing property (see \cite{lad1,lad2}) of the evolution operator $S_t$ is the main ingredient of our further
data assimilation considerations.

\begin{proposition}[Squeezing property]\label{pr:frame-for-kicks2}
 Let    \eqref{GG-cond} be in force.
Then for every $q<1$, $0<\al\le \beta<+\infty$ and $L$  there exists $N_*=N(\al,\beta,L,q)$
 such that
\begin{equation*}%\label{sq-ineq0}
\|Q_N[S_t U-S_tU_*]\|_{W_1}\le q \|U-U_*\|_{W_1},~~\forall\, t\in [\al,\beta],
,~~\forall\, N\ge N_*,
\end{equation*}
for any $U$ and $U_*$ from  the set
\[
\sD=\left\{ U\in W_2\,: \; \|S_t U\|_{W_2}\le L~~ for~all~t\in [0,\beta]\right\}.
\]
 \end{proposition}
 \begin{proof}
 The same type argument as in the proof of
   Theorem~3.5 in \cite{Chueshov-squeszing2012} leads to the desired result.
 \end{proof}

\section{Observation/measurement functionals}\label{sec:def-f}

To describe observation/measurement procedure we use  a finite family $\cL$
of linear continuous functionals $\{l_j: j=1,\ldots,N\}$ on the phase space.
If $U$ is  a phase vector which corresponds to some state of the system,
then, similar to \cite{HO-Titi-data-as2011}, we can treat the values
$\{l_j(U): j=1,\ldots,N\}$ as a set of observation data.
Our task is now to determinate the state $U$ with the help of observation functionals
$\{l_j\}$. Therefore to describe admissible  observations it is natural to involve
well-developed theory of determining functionals. This theory starts with
 the pioneering paper \cite{FP-det} on   determining modes and  was
 developed  by many authors for different classes of PDE systems
and different families of functionals (see the recent  discussion in  \cite{FJKT-det}).
 For a general theory  of the determining functionals we refer to \cite{CT-det},
see also \cite{Chueshov,cl-mem,cl-book} for a development of this theory
based on the notion of the completeness defect.
The concept  of completeness defect which was introduced
in \cite{Chu97,Chu98} seems a convenient tool in characterization of observation functionals.

\begin{definition}\label{de7.8.21}
{\rm
Let $V$ and $H$ be  reflexive Banach spaces and
$V$ is continuously and densely embedded into $H$.
{\it The completeness defect} of a set ${\cal L}$ of linear functionals
on $V$ with respect to $H$ is the value
\begin{equation}\label{7.8.21}
\epsilon_{\cal L}(V,H) =\sup\{ \parallel w \parallel_{H}  :
w\in V,\, l (w)=0,\, l\in {\cal L},\, \parallel w \parallel _{V} \le 1
\}\;.
\end{equation}
}
\end{definition}
It is  obvious
that $\epsilon_{{\cal L}_1} (V,H)\ge\epsilon_{{\cal L}_2} (V,H)$
provided Span${\cal L}_1\subset{\rm Span}{\cal L}_2$. In addition,
$\epsilon_{\cal L} (V,H)=0$ if and only if the class of functionals
${\cal L}$ is complete in $V$; this means that the property $l(w)=0$ for all
$l\in {\cal L} $ implies $w=0$.
 We can also generalize the notion of
the completeness defect by considering
 some seminorms  $\mu_V$ in \eqref{7.8.21}
instead of the norm $\|\cdot\|_H$ (see, e.g., \cite{cl-book}).
\par
Below we use the so-called interpolation operators which are related with
the set of functionals  given.
To describe their properties we need the following notion.
\begin{definition}\label{de7.8.22}
{\rm
Let $V\subset H$ be  separable Hilbert spaces and
$R$ be a linear operator from $V$ into $H$. As in \cite{Au72}  the value
$$
e^H_V (R) =\sup\{\Vert u-Ru\Vert_H\; :\; \Vert u\Vert_V\le 1\}\equiv \|I-R\|_{V\mapsto H}
$$
is said to be  the {\em global approximation error}
in $H$ arising in the approximation of
elements $v\in V$ by elements $Rv$. Here and below $\| \cdot \|_{V\mapsto H}$
denotes the operator norm for linear mappings from $V$ into $H$.
}
\end{definition}
The following assertion (see \cite{Chu98,Chueshov}
for the proof) shows that the completeness defect provides us with a bound
from below for the best possible global approximation error.
\begin{theorem}\label{th7.8.22}
Let $V$ and $H$ be the separable Hilbert spaces such that
$V$ is compactly and densely embedded into $H$.
Let ${\cal L}$ be a set of linear functionals on $V$.
Then we have the following relations
$$
\epsilon_{\cal L}(V,H)=\min\{ e^H_V (R)\; :\; R\in
{\cal R}_{\cal L}\},
$$
where ${\cal R}_{\cal L}$
is the family of linear  bounded operators $R: V\mapsto H$ and
such that $Rv=0$ for all
$v\in {\cal L}^{\perp}=\{ v\in V\; :\; l(v)=0, \; l\in {\cal
L}\}$. Moreover, we have that
\begin{equation}\label{eL-opt}
    \epsilon_{\cal L}(V,H)= e^H_V (I-Q_\cL)=\sup\{ \| Q_\cL u\|_H\; :\;\Vert u\Vert_V\le 1\},
\end{equation}
where  $Q_\cL$ is the orthoprojector in $V$ onto ${\cal L}^{\perp}$.
\end{theorem}
One can show (see \cite{Chueshov}) that any operator  $R\in {\cal R}_{\cal L}$
has the form
\begin{equation}\label{R-op}
Rv=\sum_{j=1}^N l_j(v) \psi_j, ~~~ \forall\, v\in V,
\end{equation}
where $\{\psi_j\}$ is an arbitrary finite set of elements from $V$.
This why   ${\cal R}_{\cal L}$
is called the set of interpolation operators corresponding  to the set $\cL$.
An operator $R\in \cR_\cL$ is called {\em Lagrange } interpolation operator, if it has form
\eqref{R-op} with $\{\psi_j\}$ such that $l_k(\psi_j)=\delta_{kj}$. In
the case of Lagrange operators we have that $R^2=R$, i.e., $R$ is a projector.

\par
We also note that the operator $Q_\cL$ in \eqref{eL-opt} has the following structure
\[
Q_\cL=I-P_\cL ~~~\mbox{with}~~P_\cL v=\sum_{j=1}^N (\xi_j,v)_V \xi_j, ~~~ \forall\, v\in V,
\]
where $\{\xi_j\}$ is the orthonormal basis in the orthogonal supplement  $\cM_\cL$ to the annulator
 ${\cal L}^{\perp}$ in $V$. We call $P_\cL$ the {\em optimal} interpolation operator  corresponding  to the set $\cL$.
\par

Our main example is related with the eigen-basis of the operator
$A$ defined by the form \eqref{a-form}.

\begin{example}[Modes]\label{modes}
{\rm
Denote by
${\cal L}$ the set of functionals
${\cal L}=\{ l_j(u)=(u,e_j)\; :\; j=1,2,\ldots,N\}$,
where $ \{ e_k \}$ are eigenfunctions of the operator $A$ given by the form \eqref{a-form},
see \eqref{a-form2}.
The optimal interpolation operator $P_\cL$ is Lagrange  in this case  and has the form
\begin{equation}\label{RL-opt}
P_\cL v=\sum_{j=1}^N (e_j,v) e_j, ~~~ \forall\, v\in W_1.
\end{equation}
Moreover,
$\epsilon_{\cal
L}(W_i,H)=e^H_{W_i}(P_\cL)=\la_{N+1}^{-i/2}$, $i=1,2$.
Thus the completeness defect and the global approximation error $ \|I-P_\cL\|_{W_i\mapsto H}$
can be made small after an appropriate choice of $N$.
}
\end{example}
%
%Another example is finite elements type approximations in the form presented in
%\cite{Au72}.

\section{Discrete data assimilation}\label{sec:data}
We consider
the discrete data assimilation problem in the sense due to \cite{HO-Titi-data-as2011}.
The paper \cite{HO-Titi-data-as2011} is  focused  on the case where the measurement data is taken at
a sequence of discrete times $t_n$ in contrast
with the papers \cite{Olson-titi2013,Olson-titi2003} which consider  \emph{continuous}
data assimilation. All these papers deal with for the incompressible two-dimensional Navier--Stokes equations.

Following the  idea presented in \cite{HO-Titi-data-as2011} we accept the following
definition.
\begin{definition}\label{de:data-assim}
{\rm
  Let $U(t)=S_tU_0$ be a solution to (\ref{fl.1a})--(\ref{fl.4a-id}) with initial data $U_0$ at time $t_0$.
   Let $\cL=\{l_j\}$ be a finite family of functionals on $H$ (each functional $l_j$
   is interpreted as a single observational  measurement). Let
   $R_\cL$ be some Lagrange interpolation operator related with $\cL$
   such that the sequence $\{r_\cL^n\equiv R_\cL U(t_n)\}$ represents the (joint)
   observational measurements of the reference solution $U(t)$ at a sequence $\{t_n\}$ of times, we call the sequence $\{r_\cL^n\}$  {\em observation values}.
 Now we can construct {\em prognostic values} at time $t_n$ by the formula
\begin{equation}\label{progn1}
u_n= (1-R_\cL) S_{t_n-t_{n-1}} u_{n-1}+r_\cL^n ,~~~ n=1,2,\ldots, 
\end{equation}
where $u_0$ is (unknown) vector which, according to \cite{HO-Titi-data-as2011},
 corresponds to an initial guess  of the reference
solution $U(t_0)$.
 We  can also  define the prognostic (piecewise continuous) trajectory
as
\begin{equation}\label{progn-traj}
u(t) = S_{t- t_n} u_n~~\mbox{for}~~ t \in  [t_n, t_{n+1}), ~~~ n=0,1,2,\ldots
\end{equation}
We say that the prognosis is {\em asymptotically reliable} at a sequence of times $t_n$
if
\[
\|U(t_n)-u_n\|_{W_1} \to0  ~~~\mbox{as}~~n\to +\infty.
\]
}
\end{definition}
Our goal is to find conditions on $R_\cL$, $t_n$ and $\eta$ which guarantee that the
prognosis based on a finite number of single observations  is  asymptotically reliable.
\par
We assume that $0<\alpha\le t_{n+1}-t_n\le \beta<+\infty$ for some positive $\al$ and $\beta$.
\par
The following assertion gives us a dissipativity property for prognostic values
which is important for our application of the Ladyzhenskaya squeezing property.
\begin{lemma}\label{le:dis-prognosis}
Assume that $\|U(t)\|_{W_2}\le K$ for all $t\ge t_0$.
Let
$$
\| R_\cL\|_{W_2\mapsto H}\le c_0~~\mbox{and}~~  \| 1-R_\cL\|_{H\mapsto H}\le c_1~
with~ c_1<e^{a_0\al},
$$
where $a_0$ is the constant in \eqref{S-sq1}.
Then there exists  $n_*>0$ such that
\begin{equation}\label{un-dis1}
\|u_n\|\le 1+\varrho_*~~~ \mbox{for all}~~ n\ge n_*,
\end{equation}
 where  $\varrho_*=(a_1K_G+c_0K)(1-c_1e^{-a_0\al})^{-1}$.
 If we assume in addition that  $\| 1-R_\cL\|_{W_2\mapsto W_2}\le c_2$, then
 \begin{equation}\label{un-dis2}
\|u_n\|_{W_2}\le \varrho \equiv  c_2C(\al,\beta, 1+\vr_*)+ (1+c_2) K~~~ \mbox{for all}~~ n\ge m_*\equiv 1+n_*,
\end{equation}
where $C(\al,\beta, \vr)$ is the constant from \eqref{S-sq2}.
\end{lemma}
\begin{proof}
One can see from Proposition~\ref{pr:frame-for-kicks} that
 \[
\|u_n\|\le c_1e^{-a_0\al} \|u_{n-1}\|+a_1K_G+c_0K,~~n=1,2,\ldots
\]
This implies that
  \[
\|u_n\|\le q_*^n \|u_0\|+\varrho_*,~~n=1,2,\ldots
\]
where $q_*= c_1e^{-a_0\al}$. This yields \eqref{un-dis1}.
\par
To prove \eqref{un-dis2} we note that
\[
\|u_n\|_{W_2}\le c_2 \|S_{t_n-t_{n-1}}u_{n-1}\|_{W_2}+ (1+c_2)K,~~n=1,2,\ldots
\]
Hence \eqref{un-dis2} follows from \eqref{S-sq2} and \eqref{un-dis1}.
\end{proof}
In the case of (spectral) modes we have the following assertion.
\begin{corollary}\label{co:diss}
 Let $\cL$ be the same as in Example~\ref{modes}.
 If we take in \eqref{progn1} $R_\cL$ to be  the interpolation operator given by \eqref{RL-opt},
 then there exist positive constants  $C(K_G, K,\al,\beta)$ and  $m_*$
   such that
   \begin{equation*}%\label{un-dis3}
\|u_n\|_{W_2}\le \varrho \equiv  C(K_G, K,\al,\beta)~~~ \mbox{for all}~~ n\ge m_*.
\end{equation*}
\end{corollary}
\begin{proof}
  In this case $c_0=\la_1^{-1}$ and $c_1=c_2=1$.
\end{proof}

Now we are in position to obtain the main result.

\begin{theorem}\label{th:prognosis}
Assume that $\cL$ is a finite family of functionals on $H$ and there is a Lagrange interpolation operator $R_\cL$ possessing the properties:
\begin{equation}\label{cond-WH}
 \| 1-R_\cL\|_{H\mapsto H}\le c_1~~\mbox{and}~~\| 1-R_\cL\|_{W_2\mapsto W_2}\le c_2
\end{equation}
with the constants $c_1$ and $c_2$ independent of $\cL$ such that
$c_1<e^{a_0\al}$,
where $a_0$ is the constant in \eqref{S-sq1}.
Then there exists $\epsilon_*>0$ such that under the condition $\epsilon(W_1,H)\le \epsilon_*$
the prognosis in \eqref{progn1} is asymptotically reliable for every $u_0\in H$.
\par
In the case of the modes described in Example \ref{modes} there exists $N_*$
such that the prognosis \eqref{progn1} is asymptotically reliable
with $R_\cL=P_\cL$, where $P_\cL$ is given by \eqref{RL-opt}
with some $N\ge N_*$.
\end{theorem}
  \begin{proof}
We obviously have that
    \begin{equation*}%\label{progn2}
U(t_n)-u_n= (1-R_\cL)[S_{t_n-t_{n-1}} U(t_{n-1})-  S_{t_n-t_{n-1}} u_{n-1}] ,~~~ n\ge m_*.
\end{equation*}
In the case of modes we  have that $I-R_\cL=Q_N$.
Therefore using Corollary~\ref{co:diss}  by
Proposition~\ref{pr:frame-for-kicks2}
we can choose $N_*$ such that and thus
      \begin{equation*}%\label{progn3}
\|U(t_n)-u_n\|_{W_1}= q \|U(t_{n-1})  - u_{n-1}\|_{W_1},~~~ n\ge m_*,
\end{equation*}
with $q<1$. This implies
   \[
\|U(t_n)-u_n\|_{W_1}\to0~~\mbox{as}~~ n\to+\infty
\]
with exponential speed. Therefore the statement of the theorem is valid
in the case of modes.
\par

It is obvious that under conditions \eqref{cond-WH} the hypotheses of Lemma~\ref{le:dis-prognosis}
are in force.
Therefore
in the general case Proposition~\ref{pr:frame-for-kicks2}  implies that
\begin{align}
\|S_{\Delta_n} & U(t_{n-1})-  S_{\Delta_n} u_{n-1}\|_{W_1} \notag \\ &\le  q_N \|U(t_{n-1})  - u_{n-1}\|_{W_1}
+\la_N^{1/2} \|S_{\Delta_n} U(t_{n-1})-  S_{\Delta_n} u_{n-1}\|
\label{progn4}
\end{align}
for $n\ge m_*$
with $\Delta_n =t_n-t_{n-1}$, where $q_N<1$ can be chosen as small as we need at the expense of $N$. By
Proposition~\ref{pr:Lip-H} we have that
\[
\|S_{\Delta_n} U(t_{n-1})-  S_{\Delta_n} u_{n-1}\|\le C_\beta(\vr)
\| U(t_{n-1})-   u_{n-1}\|, ~~ n\ge m_*.
\]
Since $l_j(U(t_{n-1}))=l_j(u_{n-1})$, this gives
\[
\|S_{\Delta_n} U(t_{n-1})-  S_{\Delta_n} u_{n-1}\|\le \epsilon(W_1,H)C_\beta(\vr)
\| U(t_{n-1})-   u_{n-1}\|_{W_1}, ~~ n\ge m_*.
\]
Thus \eqref{progn4} yields
\[
\|U(t_{n})-   u_{n}\|_{W_1} \le  \tilde{q} \|U(t_{n-1})  - u_{n-1}\|_{W_1}
\]
for $n\ge m_*$, where
\[
\tilde{q}=\|I-R_\cL\|_{W_1\mapsto W_1}\left[ q_N+\la^{1/2}\epsilon(W_1,H) C_\beta(\vr)\right].
\]
By the operators interpolation  from condition \eqref{cond-WH} we have that
 \[
 \|I-R_\cL\|_{W_1\mapsto W_1}\le \sqrt{c_1c_2}.
 \]
Hence we can choose $N$ and $\epsilon(W_1,H)$ such that $\tilde{q}<1$.
Therefore the prognosis is asymptotically reliable with exponential speed.
  \end{proof}
  We conclude our considerations with several remarks.
\begin{remark}\label{re:conclusion1}
{ \rm
As an example of set $\cL$ functionals $\{l_j\}$ satisfying \eqref{cond-WH}
we can consider {\em generalized modes}
which are defined by the formulas:
\[
l_j(u) =(Ke_j,u),~~~ j=1,\ldots, N,
\]
where $\{e_j\}$ is the eigen-basis of the operator $A$ and $K$ is a  linear invertible
self-adjoint operator in $H$  with maps $W_2$ into itself and is bounded in both spaces
$H$ and $W_2$. In this case the operator $R_\cL$  has the form 
\eqref{R-op} with $\psi_j=K^{-1}e_j$.
One can see that we can apply Theorem~\ref{th:prognosis} with $\al$ greater than $\frac1{a_0}\ln (1+\|K\|_{H\mapsto H})$.
}
\end{remark}
\begin{remark}\label{re:conclusion2}
{ \rm
Under the conditions of Theorem~\ref{th:prognosis} we also have that
\[
\lim_{t\to+\infty}\|U(t)-u(t)\|_{W_1}=0
 \]
for the prognostic trajectory given by \eqref{progn-traj}.
Thus the prognosis is also reliable in the sense used in  \cite{HO-Titi-data-as2011}.
}
\end{remark}

\begin{remark}\label{re:conclusion3}
{ \rm
The number of  functionals  which provides an asymptotically
reliable prognosis according to  Theorem~\ref{th:prognosis} is finite.
However the estimates for this number  which follows from the statement of theorem
are not optimal and even not constructive. The derivation of optimal bounds for the
numbers requires more careful analysis of  constants related to dissipativity
and squeezing properties of individual trajectories.
We refer to \cite{HO-Titi-data-as2011} for more constructive approach
based on the multipliers technique and 
 developed in the case of the 2D  Navier--Stokes equations for the reference solution from
 the global attractor.

}
\end{remark}

%\subsection*{Acknowledgement}

\end{document}